\documentclass[12pt,reqno]{amsart}
\usepackage{amsmath, amssymb, amsthm, enumerate, booktabs, accents, framed, graphicx, array, color, multirow, mathrsfs, amsrefs, cases, euscript, bm, multicol,subfigure}

\usepackage[utf8]{inputenc}
\usepackage[colorlinks,citecolor=blue]{hyperref}
\usepackage[usenames,dvipsnames,svgnames,table]{xcolor}
\usepackage{a4wide}

\allowdisplaybreaks[4]

\numberwithin{equation}{section}

\theoremstyle{plain}
\newtheorem{theorem}{Theorem}
\newtheorem{lemma}{Lemma}
\newtheorem{corollary}{Corollary}

\newtheorem{definition}{Definition}
\newtheorem{remark}{Remark}

\def\bE{\mathbb E}
\def\bP{\mathbb P}
\def\bN{\mathbb N}

\def\bC{\mathbb C}
\def\bM{\mathbb M}

\def\cA{\mathcal A}

\def\cH{\mathcal H}
\def\cL{\mathcal L}
\def\cX{\mathcal X}

\def\id{\mathrm {Id}}

\def\tr{\mathrm {tr}}

\def\Cov{\mathrm {Cov}}
\def\free{\mathrm {free}}

\def\sp{\mathrm {sp}}

\begin{document}

\title[Strong convergence for tensor GUE]{Strong convergence for tensor GUE random matrices}
\author[B. Collins]{Beno\^{i}t Collins}
\address{Department of Mathematics, Graduate School of Science,
Kyoto University, Kyoto 606-8502, Japan}
\email{collins@math.kyoto-u.ac.jp}

\author[W. Yuan]{Wangjun Yuan}
\address{University of Luxembourg, Department of Mathematics, Maison du Nombre 6, Avenue de la Fonte L-4364 Esch-sur-Alzette. Luxembourg.}
\email{ywangjun@connect.hku.hk}

\vspace{-2cm}
\vspace{-1cm}

\begin{abstract}
%{\red To be filled.}
Haagerup and Thorbj{\o}rnsen proved that iid GUEs converge strongly to free semicircular elements as the dimension grows to infinity. 
Motivated by considerations from quantum physics -- in particular, understanding nearest neighbor interactions in quantum spin systems --  we consider iid GUE acting on multipartite state spaces, with a mixing component on some sites and identity on the remaining sites.
We show that under proper assumptions on the dimension of the sites, strong asymptotic freeness still holds.
Our proof relies on an interpolation technology recently introduced by Bandeira, Boedihardjo and van Handel.

 \smallskip

\end{abstract}

\maketitle

\section{Introduction and statement of main result}

\subsection{Motivating considerations}
%%%%%%
A typical problem in multi-matrix random matrix theory consists in considering a sequence (indexed by a dimension) of a $d$-tuple of random matrices, consider a fixed non-commutative function in $d$ abstract variables and apply it to this sequence of random matrices, and study its properties in the large dimension limit. Multiple breakthrough related to the study of such problems have arisen from the use of free probability theory, after Voiculescu observed in \cite{Voi1991} in 1991, the asymptotic freeness of independent GUE as the dimension goes to infinity. 

In the case of asymptotic freeness, the convergence of the spectrum is primarily being considered, however, many other ``measurables'' are also considered, such as the operator norm.
This has been the motivation of strong asymptotic freeness with the first major achievement with Haagerup and Thorbj\o rnsen \cite{Haagerup2005}, and subsequently, many other variants. 
Among subsequent variants, let us mention  \cite{CMale} where it is proved that independent unitaries are strongly asymptotically free from constant matrices. 
Later, the authors of  \cites{Bordenave2019,Bordenave2023,Bordenave2024}
proved strong freeness involving random permutations and also large tensors . 
The tensor setup is actually the one in which we are interested, motivated by quantum mechanics.
 
Let us outline a motivating example based on the paper \cite{Siddhardh2019}, capturin the typical random matrix theoretical problems of interest to quantum many body system. 
We are interested in the spectrum of a matrix
\begin{align} \label{eq-sum}
    \sum_{i=1}^kX_i,
\end{align}
where $X$ is a random matrix (e.g. a GUE) acting on $M_n(\mathbb{C})^{\otimes 2}$
and $X_i$ is the version of $X$ acting on the leg $i,i+1$ of $M_n(\mathbb{C})^{\otimes k+1}$.
We refer to \eqref{eq-tilde otimes} for definitions.

It was observed by the authors of the above paper and also in \cite{Charlesworth2021}
that epsilon-freeness is the right concept for such models -- in the case where $X_i$ emanate from a different i.i.d. copy of $X$ for each $i$.
However, strong convergence for models introduced in \cite{Charlesworth2021} in the context of epsilon-freeness and of the paper \cite{Siddhardh2019} remains elusive. 
On the other hand, and quite unexpectedly, this turns out to be also an important problem in random geometry, see \cites{Magee2023}.
%quote https://arxiv.org/abs/2308.00863

It turns out that we are not able to handle the problem in such generality at that point, because techniques that are available with large tensors (developed by \cite{Bordenave2024}) do not support large rank tensor product. 
The only paper supporting in a robust way large rank tensor product is the fundamental recent paper by \cite{vH2024}.
In this paper, they obtain strong convergence for Gaussian models with completely explicit numerical bounds, which are, under some assumptions, compatible with large matrix tensor rank. 
This is the technology on which we rely. However, we had to restrict our considerations in this manuscript to the following setup:

(1) Unlike in the original model from physics, the dimensions of the various systems are not equal.

(2) Each summand in the model has to be independent from the others. 

In another vein, it is interesting to compare our manuscript with \cite{Mingo2019} 
in which it is proved that a unitary matrix and its partial transpose are asymptotically free, in the sense that we achieve freeness with less randomness than initially required by Voiculescu (in \cite{Mingo2019}, taking the transpose -- a deterministic operation -- yields freeness, whereas for us, changing a small dimension leg yields strong asymptotic freeness). The common point between \cite{Mingo2019} and our manuscript is that in both cases there is a tensor structure and a modification of the model on a subsystem.

\subsection{Statement of the results and organization}

Let $N,m$ be positive integers. We denote $[m] = \{1,2,\ldots,m\}$. For any $J \subseteq [m]$, we denote by $\{a_1,\ldots a_{|J|}\}$ the elements in $J$ with a certain order, and $|J|$ the number of elements in $J$. Let $X$ be a Gaussian unitary ensemble (GUE) matrix on $\bM_{N^{|J|}}(\bC)$. Note that $\bM_{N^{|J|}}(\bC) = \bM_N(\bC)^{\otimes |J|}$, we have the following decomposition of $X_J$:
\begin{align*}
    X = \sum_{i_1,\ldots,i_{|J|}, j_1,\ldots j_{|J|}=1}^N x_{i_1\ldots i_{|J|}j_1\ldots j_{|J|}} E_{i_1j_1} \otimes \ldots \otimes E_{i_{|J|}j_{|J|}},
\end{align*}
where the family $\{E_{ij}:1 \le i,j \le N\}$ is a basis of $\bM_N(\bC)$. Now we introduce the notation $X \tilde \otimes_J I_N^{\otimes (m-|J|)}$ as follows:
\begin{align*}
    X \tilde \otimes_J I_N^{\otimes (m-|J|)} = \sum_{i_1,\ldots,i_{|J|}, j_1,\ldots j_{|J|}=1}^N x_{i_1\ldots i_{|J|}j_1\ldots j_{|J|}} I_N \otimes \ldots \otimes E_{i_1j_1} \otimes \ldots \otimes E_{i_{|J|}j_{|J|}} \otimes \ldots \otimes I_N,
\end{align*}
where the base matrices $E_{i_1j_1}, \ldots E_{i_{|J|}j_{|J|}}$ are in the position $a_1,\ldots, a_{|J|}$ respectively. We will introduce the notation $\tilde \otimes_J$ with more details in Section \ref{sec:tensor}.

%question from Benoit: can $k$ be arbitrary, i.e. can the map i\to J_i be non-injective ? (I think that yes but can you confirm?)
%Wangjun's answer: As sets, $J_1,\ldots,J_k$ need not to be different. However, even if $J_1=J_2$, the GUEs $X_{J_1},X_{J_2}$ are different. They are independent with the same legs.
For $1 \le i \le k$, let $J_i$ be a subset of $[m]$ such that $|J_i|>m/2$. We define
\begin{align*}
    \alpha = \min_{1 \le i \le k} (2|J_i|-m).
\end{align*}
Let $X_{J_i}$ be a GUE matrix on $\bM_{N^{|J_i|}}(\bC)$, and $B_i$ be a deterministic $d \times d$ self-adjoint matrix. Consider the following model
\begin{align} \label{eq-model}
    X_N = \sum_{i=1}^k B_i \otimes \left( X_{J_i} \tilde\otimes_{J_i} I_N^{\otimes (m-|J_i|)} \right).
\end{align}

Let $\{s_1, \ldots, s_k\}$ be a family of free semicircular elements, and define
\begin{align} \label{eq-model-free}
    X_\free = \sum_{i=1}^k B_i \otimes s_i.
\end{align}

Let $\iota$ be a mapping that maps a matrix to the column vector of its entries.
We have the following result on the spectrum of $X_N$ and $X_\free$.
\begin{theorem} \label{Thm-1}
Assume that there exist parameters $\Gamma, \Theta$ that may depend on $N$, such that
\begin{align*}
    \left\| \sum_{i=1}^k \left( B_i \right)^2 \right\| \le \Gamma, \quad
    \sum_{i=1}^k \left\| \iota (B_i) \iota (B_i)^* \right\| \le \Theta.
\end{align*}
Then for all $\alpha,N,k$, for all $t \ge 0$, we have
\begin{align} \label{ineq-spectrum}
    \bP \left( \sp(X_N) \subseteq \sp(X_{\free}) + C N^{-\alpha/4} \Gamma^{1/4} \Theta^{1/4} \left( \ln^{3/4} \left( dN^m \right) + t \right) [-1,1] \right) \ge 1 - e^{-t^2},
\end{align}
and
\begin{align} \label{ineq-norm}
    \bP \left( \left\| X_N \right\| > \left\| X_{\free} \right\| + C N^{-\alpha/4} \Gamma^{1/4} \Theta^{1/4} \left( \ln^{3/4} \left( dN^m \right) + t \right) \right)
    \le e^{-t^2},
\end{align}
for a universal positive constant $C$.
\end{theorem}

If we choose $t = \ln^{3/4} (dN^m)$ in Theorem \ref{Thm-1}, we have the following corollary:
\begin{corollary} \label{Coro-1}
Suppose that the assumptions in Theorem \ref{Thm-1} hold. Then for all $\alpha,N,k$, we have
\begin{align*}
    \bP \left( \sp(X_N) \subseteq \sp(X_{\free}) + C N^{-\alpha/4} \Gamma^{1/4} \Theta^{1/4} \ln^{3/4} \left( dN^m \right) [-1,1] \right) \ge 1 - e^{-\ln^{3/2} (dN^m)},
\end{align*}
and
\begin{align*}
    \bP \left( \left\| X_N \right\| > \left\| X_{\free} \right\| + C N^{-\alpha/4} \Gamma^{1/4} \Theta^{1/4} \ln^{3/4} \left( dN^m \right) \right)
    \le e^{-\ln^{3/2} (dN^m)},
\end{align*}
for a universal positive constant $C$.
\end{corollary}

In particular, by Borel-Cantelli's lemma, we have the next corollary:
\begin{corollary} \label{Coro-2}
Suppose that the assumptions in Theorem \ref{Thm-1} hold. Assume that
\begin{align*}
    \lim_{N \to \infty} N^{-\alpha} \Gamma \Theta \ln^3 \left( dN^m \right) = 0,
\end{align*}
then for any $\epsilon>0$, we have almost surely,
\begin{align}
    \sp(X_N) \subseteq \sp(X_{\free}) + (-\epsilon,\epsilon) \quad \mathrm{and} \quad
    \left\| X_N \right\| \le \left\| X_{\free} \right\| + \epsilon,
\end{align}
eventually as $N \to \infty$.
\end{corollary}

We would like to remark that Theorem \ref{Thm-1}, Corollary \ref{Coro-1} and Corollary \ref{Coro-2} allow $k$ to grow with $N$. If we fix $k$, we have the following strong asymptotic freeness.

\begin{theorem} \label{Thm-2}
Let $k,d$ and $B_1,\ldots B_k$ to be fixed. Assume that 
\begin{align*}
    \lim_{N \to \infty} N^{-\alpha}m^3\ln^3 N = 0,
\end{align*}
then, for any non-commutative polynomial $P$, we have
\begin{align*}
    \lim_{N \to \infty} \left\| P \left( X_{J_1} \tilde\otimes_{J_1} I_N^{\otimes (m-|J_1|)}, \ldots, X_{J_k} \tilde\otimes_{J_k} I_N^{\otimes (m-|J_k|)} \right) \right\|
    = \left\| P \left( s_1,\ldots,s_k \right) \right\|
\end{align*}
almost surely, where $\{s_1, \ldots, s_k\}$ be a family of free semicircular elements.
\end{theorem}

The rest of the paper is organized as follows. We introduce some preliminaries on tensors, free probability, and matrices in Section \ref{sec:preliminary}. Then we present the proof of Theorem \ref{Thm-1} and Theorem \ref{Thm-2} in Section \ref{sec:proof 1} and Section \ref{sec:proof 2} respectively.

\subsection{Acknowledgements}
B. C. is supported by JSPS Grant-in-Aid Scientific Research (B) no. 21H00987, and Challenging Research (Exploratory) no. 20K20882 and 23K17299.

W. Y. gratefully acknowledges the financial support of ERC Consolidator Grant 815703 ”STAMFORD: Statistical Methods for High Dimensional Diffusions”. 
%%%add funding information there. 

\section{Preliminaries} \label{sec:preliminary}

\subsection{Tensor products} \label{sec:tensor}

For two vector spaces $\cA_1$ and $\cA_2$, we define $\cA_1 \otimes \cA_2$ as the tensor  product of $\cA_1$ and $\cA_2$. Let $m,N \in \bN$, we consider the tensor $\cH^{\otimes m}$ of $N$ dimensional vector space $\cH$. For any permutation $\sigma$ on $[m]$, we consider the linear transformation $\cA_\sigma: \cH^{\otimes m} \to \cH^{\otimes m}$ given by
\begin{align*}
    \cA_\sigma \left( x_1 \otimes \ldots \otimes x_m \right)
    = x_{\sigma(1)} \otimes \ldots \otimes x_{\sigma(m)}, \quad \forall x_1, \ldots, x_m \in \cH.
\end{align*}
Then $\cA_\sigma$ is a unitary transformation. For any linear transformation $\cX: \cH^{\otimes m} \to \cH^{\otimes m}$, we consider the linear transformation $\cA_{\sigma} \cX \cA_{\sigma^{-1}}$. If $\cX$ is of the form $\cX = \cX_1 \otimes \ldots \otimes \cX_m$, where $\cX_i$ are linear transformation on $\cH$, then for any $x_1, \ldots, x_m \in \cH$, we have
\begin{align*}
    &\cA_{\sigma} \cX \cA_{\sigma^{-1}} \left( x_1 \otimes \ldots \otimes x_m \right)
    = \cA_{\sigma} \cX \left( x_{\sigma^{-1}(1)} \otimes \ldots \otimes x_{\sigma^{-1}(m)} \right) \\
    &= \cA_{\sigma} \left( \cX_1\left(x_{\sigma^{-1}(1)}\right) \otimes \ldots \otimes \cX_m\left(x_{\sigma^{-1}(m)}\right) \right)
    = \cX_{\sigma(1)}(x_1) \otimes \ldots \otimes \cX_{\sigma(m)}(x_m),
\end{align*}
which implies that $\cA_{\sigma} \cX \cA_{\sigma^{-1}} = \cX_{\sigma(1)} \otimes \ldots \otimes \cX_{\sigma(m)}$. Therefore, we can define $\cL_\sigma$ the linear transformation on the space of linear transformation on $\cH^{\otimes m}$ given by $\cL_\sigma(\cX) = \cA_{\sigma} \cX \cA_{\sigma^{-1}}$. We would like to remark that, if a basis of $\cH$ is fixed, then any vector in $\cH$ corresponds to a vector in $\bC^N$, and a linear transformation on $\cH^{\otimes m}$ is equivalent to a matrix in $\bM_N(\bC)^{\otimes m} = \bM_{N^m}(\bC)$. Hence, we may abuse the notation and view $\cL_\sigma$ as a linear transformation on matrix space $\bM_N(\bC)^{\otimes m}$.

Let $J = \{a_1,\ldots,a_{|J|}\}$ be an (ordered) subset of $[m]$, and let $X \in \bM_{N^{|J|}}(\bC) = \bM_N(\bC)^{\otimes |J|}$ and $Y \in \bM_{N^{m-|J|}}(\bC) = \bM_N(\bC)^{\otimes (m-|J|)}$. Then $X \otimes Y$ is an element belonging to $\bM_N(\bC)^{\otimes m}$.  We choose the permutation $\sigma$ on $[m]$ such that $\sigma(l) = a_l$ for $1 \le l \le |J|$. Then we define
\begin{align} \label{eq-tilde otimes}
    X \tilde \otimes_J Y = \cL_\sigma (X \otimes Y).
\end{align}

\begin{remark}
If $J = \{1, \ldots, |J|\}$ and $\sigma = \id_{[m]}$, then the operator $\tilde \otimes_J$ becomes the usual operator $\otimes$.
\end{remark}

\begin{remark}
The linear transformation $\cL_\sigma$ depends on $\sigma$. However, in our paper, we will use this notation for $Y = \id_N^{\otimes (m-|J|)}$. Hence, $L_\sigma$ only depends on $J$, but not the value of $\sigma(|J|+1), \ldots, \sigma(m)$.
\end{remark}

We have the following property:
\begin{lemma} \label{Lem-1}
For $X,Y \in \bM_{N^{|J|}}(\bC), Z,W \in \bM_{N^{m-|J|}}(\bC)$, $a,b,c,d \in \bC$, we have the following statements.
\begin{enumerate}
    \item [(1)] $(aX+bY) \tilde \otimes_J (cZ+dW) = ac X \tilde \otimes_J Z + ad X \tilde \otimes_J W + bc Y \tilde \otimes_J Z + bd Y \tilde \otimes_J W$.
    \item [(2)] $(X \tilde \otimes_J Z) (Y \tilde \otimes_J W) = (XY) \tilde \otimes_J (ZW)$.
\end{enumerate}
\end{lemma}

\begin{proof}
The proof is a direct verification using the definition above, noting that the definition does not depend on the representation.
\end{proof}

Next, we turn to the relationship between $\tilde \otimes_J, \otimes$ and $\iota$. We have the following lemma.

\begin{lemma} \label{Lem-2}
Let $X \in \bM_{N^{|J|}}(\bC), Y \in \bM_{N^{m-|J|}}(\bC)$, then there exists a permutation matrix $U \in \bM_{N^m}(\bC)$ that only depends on $J$, such that
\begin{align} \label{eq-matrix-vector}
      \iota \left( X \tilde \otimes_J Y \right) = U \left( \iota(X) \otimes \iota(Y) \right).
\end{align}
\end{lemma}

\begin{proof}
(1). Using the fact that $A \otimes B$ and $B \otimes A$ have the same sets of entries for any $A$ and $B$, one can check with the definition that $X \tilde \otimes_J Y$ and $X \otimes Y$ have the same sets of entries, and the positions of the entries depends only on $J$. Note that the set of entries of $\iota(X \otimes Y)$ is the same as the set of components of the vector $\iota(X) \otimes \iota(Y)$. Thus, there exists a permutation matrix $U$ that only depends on $J$, such that \eqref{eq-matrix-vector} holds.
\end{proof}

\subsection{Free probability} \label{sec:free probability}

In this subsection, we recall some definitions and results in free probability. We refer interested readers to \cites{MingoSpeicher2017,NicaSpeicher} for more details on free probability.

\begin{definition}
Let $\cA$ be a unital associative algebra over $\bC$ and $\tau:\cA \to \bC$ be a linear functional satisfying $\tau(1)=1$. Then, the pair $(\cA,\tau)$ is called a \emph{non-commutative probability space}.
\end{definition}

\begin{definition}
Let $(\cA,\tau)$ be a non-commutative probability space.
\begin{enumerate}
    \item If $\tau(ab)=\tau(ba)$ for all $a,b \in \cA$, then $(\cA, \tau)$ a \emph{tracial non-commutative probability space}.
    \item If $\cA$ is a $*$-algebra and $\tau(a^*a)\ge0$ for all $a \in \cA$, then we call $\tau$ a \emph{state} and $(\cA,\tau)$ a \emph{$*$-probability space}. In addition, if $\tau(a^*a)=0$ implies that $a=0$, then we say $\tau$ is \emph{faithful}.
\end{enumerate}
\end{definition}

\begin{definition}
Let $(\cA,\tau)$ be a non-commutative probability space ($*$-probability space, resp.) and $I$ be an index set. Then the sub-algebras $\{\cA_i:i \in I\}$ of $\cA$ are \emph{free} in $(\cA,\tau)$ if
\begin{align*}
    \tau \left( \left( P_1(a_{i_1}) - \tau \left( P_1(a_{i_1}) \right) \right) \ldots \left( P_k(a_{i_k}) - \tau \left( P_k(a_{i_k}) \right) \right) \right) = 0,
\end{align*}
for all
\begin{itemize}
    \item $k \in \bN$,
    \item $i_1, \ldots, i_k \in I$ satisfying $i_1 \not= i_2 \not= i_3 \not= \ldots \not= i_{k-1} \not= i_k$, 
    \item polynomials (*-polynomials, resp.) $P_1, \ldots, P_k$,
    \item $a_j \in \cA_{i_j}$ for all $1 \le j \le k$.
\end{itemize}
A family of elements $\{a_i\}_{i \in I}$ in $\cA$ are called \emph{free} if the unital sub-algebras generated by $\{1,a_i\}$ ($\{1,a_i,a_i^*\}$, resp.) are free.
\end{definition}

\begin{definition}
Let $(\cA_n,\tau_n)$ be a sequence of non-commutative probability space (resp. a $*$-probability space). A sequence $a_{n,1},\ldots,a_{n,k} \in \cA_n$ is \emph{asymptotically free} if
\begin{align*}
    \tau_n \left( \left( P_1(a_{n,i_1}) - \tau \left( P_1(a_{n,i_1}) \right) \right) \ldots \left( P_m(a_{n,i_m}) - \tau \left( P_m(a_{n,i_m}) \right) \right) \right) \to 0,
\end{align*}
as $n \to \infty$, for any $m \in \bN$, $1 \le i_1, \ldots, i_m \le k$ satisfying $i_1 \not= i_2 \not= i_3 \not= \ldots \not= i_{m-1} \not= i_m$, and any polynomials (*-polynomials, resp.) $P_1, \ldots, P_k$.
\end{definition}

\begin{definition}
Let $(\cA, \tau)$ be a *-probability space. A self-adjoint element $s \in \cA$ is called a semicircular element if it has moments
\begin{align*}
    \tau(s^k) = \begin{cases}
        0, & k = 2m-1 \mathrm{\ for \ some \ } m \in \bN, \\
        \frac{1}{m+1} \binom{2m}{m}, &  k = 2m \mathrm{\ for \ some \ } m \in \bN.
    \end{cases}
\end{align*}
\end{definition}

\begin{remark}
The moments of a semicircular element $s$ are
\begin{align*}
    \tau \left( s^k \right)
    = \dfrac{1}{2\pi} \int_{-2}^2 t^k \sqrt{4-t^2} dt, \quad \forall k \in \bN.
\end{align*}
\end{remark}

Next, we introduce the weak convergence in the framework of free probability.

\begin{definition}
Let $(\cA_n,\tau_n)$ for every $n \in \bN$ as well as $(\cA,\tau)$ be non-commutative probability spaces. Let $\{b_n:n\in\bN\}$ be a sequence of elements with $b_{n,1}, \ldots, b_{n,k} \in \cA_n$ for every $n \in N$ and let $b_1, \ldots, b_k \in \cA$. We say that the family $\{b_{n,1}, \ldots, b_{n,k}:n\in\bN\}$ converges in distribution to $\{b_1,\ldots,b_k\}$ if it holds that
\begin{align*}
    \lim_{n \to \infty} \tau_n \left( b_{n,i_1}^{a_1} \ldots b_{n,i_m}^{a_m} \right)
    = \tau \left( b_{i_1}^{a_1} \ldots b_{i_m}^{a_m} \right), \quad \forall m \in \bN, \forall 1\le i_1,\ldots,i_m \le k, \forall a_1,\ldots, a_m \in \{1,*\}.
\end{align*}
\end{definition}

Now we turn to the so-called $C^*$-probability space.

\begin{definition}
Let $(\cA,\tau)$ be a $*$-probability space endowed with a norm $\|\cdot\|: \cA \to [0,+\infty)$ which makes it a complete normed vector space and which satisfies
\begin{align*}
    \|ab\| \le \|a\| \|b\| \quad \mathrm{and} \quad \|a^*a\| = \|a\|^2, \forall a,b\in \cA.
\end{align*}
Then $(\cA,\tau)$ is a \emph{$C^*$-probability space}.
\end{definition}

In the following, we collect some asymptotic freeness results on Gaussian matrices from \cite{BBvH2023}, which will be used later.

For each $N \in \bN$, let $\{H_i^N:1 \le i \le m\}$ be a family of independent self-adjoint $N$-dimensional random matrices satisfying $\bE [H_i^N] = 0$ and $\bE [(H_i^N)^2] = I_N$ for all $1 \le i \le m$, and have the following decomposition
\begin{align*}
    H_i^N = \sum_{j=1}^{N_i} g_{i,j}^N B_{ij}^N,
\end{align*}
for some $N_i \in \bN$, where $\{g_{i,j}^N:1 \le i \le m, 1 \le j \le N_i\}$ is a family of i.i.d. standard real Gaussian random variables, and $\{B_{ij}^N:1 \le i \le m, 1 \le j \le N_i\}$ is a family of $N$-dimensional self-adjoint complex matrices. Let $\{A_i:1 \le i \le m\}$ be a family of $d$-dimensional self-adjoint matrices. Define
\begin{align*}
    &\Theta_N = \sum_{i=1}^m A_i \otimes H_i^N = \sum_{i=1}^m A_i \otimes \sum_{j=1}^{N_i} g_{i,j}^N B_{ij}^N, \\
    & \Theta_{N,\free} = \sum_{i=1}^m A_i \otimes \sum_{j=1}^{N_i} B_{ij}^N \otimes s_{i,j}, \\
    &\Theta_{\free} = \sum_{i=1}^m A_i \otimes s_i,
\end{align*}
where $\{s_1, \ldots, s_k\}$ and $\{s_{i,j}:1 \le i \le k, 1 \le j \le N_i\}$ are families of free semicircular elements.

The following lemma is the linearization argument of \cite{Haagerup2005} and is a direct consequence of \cite{BBvH2023}*{Theorem 7.7}.

\begin{lemma} (\cite{BBvH2023}*{Theorem 7.7}) \label{lem-BBvH 7.7}
Assume that the following holds
\begin{align*}
    \sp \left( \Theta_N \right) \subseteq \sp \left( \Theta_\free \right) + [-\epsilon,\epsilon], \quad a.s.
\end{align*}
eventually as $N \to \infty$, for any $\epsilon>0$. Then for any non-commutative polynomial $P$, the following holds:
\begin{align*}
    \limsup_{N \to \infty} \left\| P \left( H_1^N, \ldots, H_m^N \right) \right\| \le \left\| P \left( s_1, \ldots, s_m \right) \right\|, \quad a.s.
\end{align*}
\end{lemma}

The following lemma is \cite{BBvH2023}*{Lemma 7.9}.
\begin{lemma} (\cite{BBvH2023}*{Lemma 7.9}) \label{lem-BBvH 7.9}
For every $N \in \bN$, we have
\begin{align*}
    (\tr \otimes \tau) \left( H_{j_1}^{N,\free} \ldots H_{j_q}^{N,\free} \right)
    = \tau \left(  s_{j_1} \ldots s_{j_q} \right), \quad \forall q \in \bN, \ \forall 1 \le j_1, \ldots, j_q \le m,
\end{align*}
where $H_i^{N,\free} = \sum_{j=1}^{N_i} B_{ij}^N \otimes s_{i,j}$. In particular, it holds that
\begin{align*}
    \sp \left( \Theta_{N,\free} \right) = \sp \left( \Theta_\free \right).
\end{align*}
\end{lemma}

The next lemma is a combination of part a of \cite{BBvH2023}*{Theorem 2.10} and \cite{BBvH2023}*{Lemma 7.11}, which yields the weak asymptotic freeness.
\begin{lemma} (\cite{BBvH2023}*{Theorem 2.10}, \cite{BBvH2023}*{Lemma 7.11}) \label{lem-BBvH 2.10}
Let the control parameter $v$ be defined in \eqref{def-v}. If for all $1 \le i \le m$, $v(H_i^N) = o_N(\ln^{-3/2} N)$, then for any non-commutative polynomial $P$, it holds almost surely that
\begin{align*}
    \lim_{N \to \infty} \tr \left( P \left( H_1^N, \ldots, H_m^N \right) \right)
    = \tau \left( P \left( s_1,\ldots,s_m \right) \right).
\end{align*}
\end{lemma}

\subsection{Estimates on spectrum and operator norm}

For $N,k \in \bN$, let $A_0,\ldots,A_k \in \bM_N(\bC)$ be self-adjoint matrices, let $\{g_1,\ldots,g_k\}$ be a family of i.i.d. real standard Gaussian variables, and let $\{s_1,\ldots,s_k\}$ be a family of free semicircular elements. We define
\begin{align*}
    X_N = A_0 + \sum_{i=1}^k g_i A_i, \quad\quad
    X_{\free} = A_0 \otimes {\bf 1} + \sum_{i=1}^k A_i \otimes s_i.
\end{align*}
We introduce the following control parameters:
\begin{align} \label{def-sigma}
    \sigma(X_N)^2 = \left\| \sum_{i=1}^k \left( A_i \right)^2 \right\|
\end{align}
and
\begin{align} \label{def-v}
    v(X_N)^2 = \left\| \Cov(\iota(X_N)) \right\|
    = \left\| \sum_{i=1}^k \iota(A_i) \iota(A_i)^* \right\|.
\end{align}
We define
\begin{align} \label{def-u}
    u(X_N)^2 = \sigma(X_N) v(X_N).
\end{align}

The following lemma is a direct consequence of \cite{BBvH2023}*{Theorem 2.1, Corollary 2.2}.

\begin{lemma} (\cite{BBvH2023}*{Theorem 2.1, Corollary 2.2}) \label{Lem:vanHandel}
For all $t \ge 0$, we have
\begin{align*}
    \bP \left( \sp(X_N) \subseteq \sp(X_{\free}) + C u(X_N) \left( \ln^{3/4} N + t \right) [-1,1] \right) \ge 1 - e^{-t^2},
\end{align*}
and
\begin{align*}
    \bP \left( \left\| X_N \right\| > \left\| X_\free \right\| + C u(X_N) \left( \ln^{3/4} N + t \right) \right) \le e^{-t^2},
\end{align*}
for a universal positive constant $C$.
\end{lemma}

\subsection{Matrix inequalities}

The following lemma is the matrix version of Cauchy-Schwarz inequality.
\begin{lemma} \label{lem-matrix CS}
For any $n \in \bN$, any family $\{E_i,F_i:1\le i \le n\}$ of $N \times N$ matrices, we have
\begin{align*}
    \left\| \sum_{i=1}^n E_iF_i \right\|^2
    \le \left\| \sum_{i=1}^n E_iE_i^* \right\| \left\| \sum_{i=1}^n F_i^*F_i \right\|.
\end{align*}
\end{lemma}

\begin{proof}
By Cauchy-Schwarz inequality, for any $x \in \bC^N$ with $\|x\|=1$, any family $\{E_i,F_i:1\le i \le n\}$, we obtain
\begin{align*}
    x^* \left( \sum_{i=1}^n E_iF_i \right) y
    = \sum_{i=1}^n x^* E_iF_i y
    \le \left( \sum_{i=1}^n x^* E_iE_i^* x \right)^{1/2} \left( \sum_{i=1}^n y^* F_i^*F_i y \right)^{1/2}.
\end{align*}
Hence,
\begin{align*}
    &\left\| \sum_{i=1}^nE_iF_i \right\|
    = \sup_{x,y \in \bC^N, \|x\|=\|y\|=1} x^* \left( \sum_{i=1}^n E_iF_i \right) y \\
    &\le \sup_{x,y \in \bC^N, \|x\|=\|y\|=1} \left( \sum_{i=1}^n x^* E_iE_i^* x \right)^{1/2} \left( \sum_{i=1}^n y^* F_i^*F_i y \right)^{1/2}
    = \left\| \sum_{i=1}^n E_iE_i^* \right\|^{1/2} \left\| \sum_{i=1}^n F_i^*F_i \right\|^{1/2}.
\end{align*}
\end{proof}

\section{Proof of Theorem \ref{Thm-1}} \label{sec:proof 1}

We divide the proof into three steps.

{\bf Step 1.} In this step, we introduce some notation.
Let $\{e_i^{(N)}:1 \le i \le N\}$ be the canonical basis of $\bC^N$. For $1 \le i \le N$, let
\begin{align*}
    E_{ii}^{(N)} = e_i^{(N)} \left( e_i^{(N)} \right)^*,
\end{align*}
and for $1 \le i < j \le N$ let
\begin{align*}
    & E_{ij}^{(N)} = \dfrac{1}{\sqrt{2}} \left( e_i^{(N)} \left( e_j^{(N)} \right)^* + e_j^{(N)} \left( e_i^{(N)} \right)^* \right), \\
    & \tilde E_{ij}^{(N)} = \dfrac{\sqrt{-1}}{\sqrt{2}} \left( e_i^{(N)} \left( e_j^{(N)} \right)^* - e_j^{(N)} \left( e_i^{(N)} \right)^* \right).
\end{align*}
Recall that for each $1 \le i \le k$, $X_{J_i}$ is a GUE matrix of dimension $N^{|J_i|}$. We have the decomposition
\begin{align*}
    X_{J_i} = N^{-|J_i|/2} \sum_{j=1}^{N^{2|J_i|}} g_{i,j} A_{i,j},
\end{align*}
where $\{A_{i,j}: 1 \le j \le N^{2|J_i|}\}$ is the set
\begin{align*}
    \{E_{ll}^{(N^{|J_i|})},E_{rs}^{(N^{|J_i|})},\tilde E_{rs}^{(N^{|J_i|})}:1 \le l \le N^{|J_i|}, 1 \le r<s \le N^{|J_i|}\},
\end{align*}
and $\{g_{i,j}: 1 \le i \le k, 1 \le j \le N^{2|J_i|}\}$ is a family of i.i.d. real standard Gaussian variables.
Then the matrix $X_{J_i} \tilde \otimes_{J_i} I_N^{\otimes (m-|J_i|)}$ has the following representation
\begin{align*}
    X_{J_i} \tilde \otimes_{J_i} I_N^{\otimes (m-|J_i|)}
    = N^{-|J_i|/2} \sum_{j=1}^{N^{2|J_i|}} g_{i,j} A_{i,j} \tilde \otimes_{J_i} I_N^{\otimes (m-|J_i|)}.
\end{align*}
Thus, by \eqref{eq-model}, we can write
\begin{align*}
    X_N &= \sum_{i=1}^k B_i \otimes \left( \sum_{j=1}^{N^{2|J_i|}} N^{-|J_i|/2} g_{i,j} A_{i,j} \tilde \otimes_{J_i} I_N^{\otimes (m-|J_i|)} \right) \\
    &= \sum_{i=1}^k N^{-|J_i|/2} \sum_{j=1}^{N^{2|J_i|}} g_{i,j} B_i \otimes \left( A_{i,j} \tilde \otimes_{J_i} I_N^{\otimes (m-|J_i|)} \right).
\end{align*}

{\bf Step 2.} Let $X_{N,\free}$ be the object obtained from $X_N$ by replacing the i.i.d. standard real Gaussian random variables by free semicircular elements. That is,
\begin{align*}
    X_{N,\free}
    = \sum_{i=1}^k N^{-|J_i|/2} \sum_{j=1}^{N^{2|J_i|}} B_i \otimes \left( A_{i,j} \tilde \otimes_{J_i} I_N^{\otimes (m-|J_i|)} \right) \otimes s_{i,j},
\end{align*}
where $\{s_{i,j}:1 \le i \le k, 1 \le j \le N^{2|J_i|}\}$ is a family of free semicircular elements.
In this step, we will show the following lemma.
\begin{lemma}
Under assumptions in Theorem \ref{Thm-1}, we have
then for all $t \ge 0$, we have
\begin{align*}
    \bP \left( \sp(X_N) \subseteq \sp(X_{N,\free}) + C N^{-\alpha/4} \Gamma^{1/4} \Theta^{1/4} \left( \ln^{3/4} \left( dN^m \right) + t \right) [-1,1] \right) \ge 1 - e^{-t^2},
\end{align*}
and
\begin{align*}
    \bP \left( \left\| X_N \right\| > \left\| X_{N,\free} \right\| + C N^{-\alpha/4} \Gamma^{1/4} \Theta^{1/4} \left( \ln^{3/4} \left( dN^m \right) + t \right) \right)
    \le e^{-t^2},
\end{align*}
for a universal positive constant $C$.
\end{lemma}

\begin{proof}
By the definition \eqref{def-sigma}, we obtain
\begin{align} \label{ineq-sigma-1}
    \sigma(X_N)^2 &= \left\| \sum_{i=1}^k N^{-|J_i|} \sum_{j=1}^{N^{2|J_i|}} \left( B_i \otimes \left( A_{i,j} \tilde \otimes_{J_i} I_N^{\otimes (m-|J_i|)} \right) \right)^2 \right\| \nonumber \\
    & = \left\| \sum_{i=1}^k N^{-|J_i|} \left( B_i \right)^2 \otimes \sum_{j=1}^{N^{2|J_i|}} \left( A_{i,j} \tilde \otimes_{J_i} I_N^{\otimes (m-|J_i|)} \right)^2 \right\| \nonumber \\
    & = \left\| \sum_{i=1}^k N^{-|J_i|} \left( B_i \right)^2 \otimes \left( \sum_{j=1}^{N^{2|J_i|}} \left( A_{i,j} \right)^2 \tilde \otimes_{J_i} I_N^{\otimes (m-|J_i|)} \right) \right\|,
\end{align}
where we use Lemma \ref{Lem-1} in the third line.
For $1 \le i \le k$, we have
\begin{align*}
    \left( E_{ll}^{(N^{|J_i|})} \right)^2 = E_{ll}^{(N^{|J_i|})}, \forall 1 \le l \le N^{|J_i|},
\end{align*}
and
\begin{align*}
    \left( E_{rs}^{(N^{|J_i|})} \right)^2
    = \left( \tilde E_{rs}^{(N^{|J_i|})} \right)^2
    = \dfrac{1}{2} \left( E_{rr}^{(N^{|J_i|})} + E_{ss}^{(N^{|J_i|})} \right),
    \quad \forall 1 \le r<s \le N^{|J_i|}.
\end{align*}
Hence, by Lemma \ref{Lem-1}, we obtain
\begin{align} \label{ineq-sigma-3}
    \sum_{j=1}^{N^{2|J_i|}} \left( A_{i,j} \right)^2 \tilde \otimes_{J_i} I_N^{\otimes (m-|J_i|)}
    =& \left( \sum_{j=1}^{N^{2|J_i|}} \left( A_{i,j} \right)^2 \right) \tilde \otimes_{J_i} I_N^{\otimes (m-|J_i|)} \nonumber \\
    =& N^{|J_i|} I_{N^{|J_i|}} \tilde \otimes_{J_i} I_N^{\otimes (m-|J_i|)}
    = N^{|J_i|} I_{N^m}.
\end{align}
Substituting \eqref{ineq-sigma-3} to \eqref{ineq-sigma-1}, we get
\begin{align} \label{ineq-sigma}
    \sigma(X_N)^2
    = \left\| \sum_{i=1}^k \left( B_i \right)^2 \otimes I_{N^m} \right\|
    = \left\| \sum_{i=1}^k \left( B_i \right)^2 \right\|
    \le \Gamma.
\end{align}

By the definition \eqref{def-v}, we have
\begin{align} \label{ineq-v-1}
    &v(X_N)^2 = \left\| \sum_{i=1}^k N^{-|J_i|} \sum_{j=1}^{N^{2|J_i|}} \iota \left( B_i \otimes \left( A_{i,j} \tilde \otimes_{J_i} I_N^{\otimes (m-|J_i|)} \right) \right) \iota \left( B_i \otimes \left( A_{i,j} \tilde \otimes_{J_i} I_N^{\otimes (m-|J_i|)} \right) \right)^* \right\| \nonumber \\
    & = \left\| \sum_{i=1}^k N^{-|J_i|} \sum_{j=1}^{N^{2|J_i|}} \left( \iota \left( B_i \right) \otimes \iota \left( A_{i,j} \tilde \otimes_{J_i} I_N^{\otimes (m-|J_i|)} \right) \right) \left( \iota \left( B_i \right) \otimes \iota \left( A_{i,j} \tilde \otimes_{J_i} I_N^{\otimes (m-|J_i|)} \right) \right)^* \right\| \nonumber \\
    & = \left\| \sum_{i=1}^k N^{-|J_i|} \left( \iota (B_i) \iota (B_i)^* \right) \otimes \sum_{j=1}^{N^{2|J_i|}} \iota \left( A_{i,j} \tilde \otimes_{J_i} I_N^{\otimes (m-|J_i|)} \right) \iota \left( A_{i,j} \tilde \otimes_{J_i} I_N^{\otimes (m-|J_i|)} \right)^* \right\|,
\end{align}
where we use Lemma \ref{Lem-2} in the second line with the permutation matrix being absorbed by the operator norm.

Note that
\begin{align*}
    \iota \left( E_{ll}^{(N^{|J_i|})} \right) \iota \left( E_{ll}^{(N^{|J_i|})} \right)^*
    = E_{l'l'}^{(N^{2|J_i|})},
\end{align*}
and
\begin{align*}
    \iota \left( E_{rs}^{(N^{|J_i|})} \right) \iota \left( E_{rs}^{(N^{|J_i|})} \right)^*
    + \iota \left( \tilde E_{rs}^{(N^{|J_i|})} \right) \iota \left( \tilde E_{rs}^{(N^{|J_i|})} \right)^*
    = E_{r'r'}^{(N^{2|J_i|})} + E_{s's'}^{(N^{2|J_i|})},
\end{align*}
where $l'$ is the position of the 1 in the vector $\iota \left( E_{ll}^{(N^{|J_i|})} \right)$, and $r',s'$ are the positions of the 1 in the vectors $\iota \left( e_r^{(N^{|J_i|})} \left( e_s^{(N^{|J_i|})} \right)^* \right)$ and $\iota \left( e_s^{(N^{|J_i|})} \left( e_r^{(N^{|J_i|})} \right)^* \right)$.
Hence, we have
\begin{align} \label{ineq-v-2}
    \sum_{j=1}^{N^{2|J_i|}} \iota \left( A_{i,j}  \right) \iota \left( A_{i,j} \right)^* = I_{N^{2|J_i|}}.
\end{align}
Thus, by Lemma \ref{Lem-2}, there exists a permutation matrix $U_i$, such that
\begin{align} \label{eq-v-4}
    & \sum_{j=1}^{N^{2|J_i|}} \iota \left( A_{i,j} \tilde \otimes_{J_i} I_N^{\otimes (m-|J_i|)} \right) \iota \left( A_{i,j} \tilde \otimes_{J_i} I_N^{\otimes (m-|J_i|)} \right)^* \nonumber \\
    & = \sum_{j=1}^{N^{2|J_i|}} U_i \left( \iota \left( A_{i,j} \right) \otimes \iota \left( I_N^{\otimes (m-|J_i|)} \right) \right) \left( \iota \left( A_{i,j} \right) \otimes \iota \left( I_N^{\otimes (m-|J_i|)} \right) \right)^* U_i^* \nonumber \\
    & = U_i \left( \left( \sum_{j=1}^{N^{2|J_i|}} \iota \left( A_{i,j} \right) \iota \left( A_{i,j} \right)^* \right) \otimes \left( \iota \left( I_N^{\otimes (m-|J_i|)} \right)  \iota \left( I_N^{\otimes (m-|J_i|)} \right)^* \right) \right) U_i^* \nonumber \\
    & = U_i \left( I_{N^{2|J_i|}} \otimes \left( \iota \left( I_N^{\otimes (m-|J_i|)} \right)  \iota \left( I_N^{\otimes (m-|J_i|)} \right)^* \right) \right) U_i^*,
\end{align}
where we substitute \eqref{ineq-v-2} in the last line.

Besides, applying Lemma \ref{Lem-2}, we obtain
\begin{align} \label{ineq-v-3}
    &\left\| \iota \left( I_N^{\otimes (m-|J_i|)} \right) \iota \left( I_N^{\otimes (m-|J_i|)} \right)^* \right\|
    = \left\| \iota \left( I_N \right)^{\otimes (m-|J_i|)} \left( \iota \left( I_N \right)^{\otimes (m-|J_i|)} \right)^* \right\| \nonumber \\
    &= \left\| \left( \iota \left( I_N \right) \iota \left( I_N \right)^* \right)^{\otimes (m-|J_i|)} \right\|
    = \left\| \iota \left( I_N \right) \iota \left( I_N \right)^* \right\|^{m-|J_i|}
    = N^{m-|J_i|},
\end{align}
where the permutation matrix is absorbed by the operator norm.
Substituting \eqref{eq-v-4} and \eqref{ineq-v-3} to \eqref{ineq-v-1}, we obtain
\begin{align} \label{ineq-v}
    v(X_N)^2
    & \le \sum_{i=1}^k N^{-|J_i|} \left\| \iota (B_i) \iota (B_i)^* \right\| \left\| \sum_{j=1}^{N^{2|J_i|}} \iota \left( A_{i,j} \tilde \otimes_{J_i} I_N^{\otimes (m-|J_i|)} \right) \iota \left( A_{i,j} \tilde \otimes_{J_i} I_N^{\otimes (m-|J_i|)} \right)^* \right\| \nonumber \\
    & \le \sum_{i=1}^k N^{m-2|J_i|} \left\| \iota (B_i) \iota (B_i)^* \right\|
    \le N^{-\alpha} \Theta.
\end{align}
where we use triangle inequality in the first line.

By the definition \eqref{def-u}, we combine \eqref{ineq-sigma} and \eqref{ineq-v} to obtain 
\begin{align} \label{ineq-u}
    u(X_N) \le N^{-\alpha/4} \Gamma^{1/4} \Theta^{1/4}.
\end{align}
Therefore, applying Lemma \ref{Lem:vanHandel}, we have
\begin{align*}
    &\bP \left( \sp(X_N) \subseteq \sp(X_{N,\free}) + C N^{-\alpha/4} \Gamma^{1/4} \Theta^{1/4} \left( \ln^{3/4} \left( dN^m \right) + t \right) [-1,1] \right) \\
    &\ge \bP \left( \sp(X_N) \subseteq \sp(X_{N,\free}) + C u(X_N) \left( \ln^{3/4} \left( dN^m \right) + t \right) [-1,1] \right)
    \ge 1 - e^{-t^2},
\end{align*}
and
\begin{align*}
    &\bP \left( \left\| X_N \right\| > \left\| X_{N,\free} \right\| + C N^{-\alpha/4} \Gamma^{1/4} \Theta^{1/4} \left( \ln^{3/4} \left( dN^m \right) + t \right) \right) \\
    &\le \bP \left( \left\| X_N \right\| > \left\| X_{N,\free} \right\| + C u(X_N) \left( \ln^{3/4} \left( dN^m \right) + t \right) \right)
    \le e^{-t^2},
\end{align*}
for any $t \ge 0$.
\end{proof}

{\bf Step 3.} In the last step, we conclude the proof via Lemma \ref{lem-BBvH 7.9}.

Note that for any $1 \le i \le k$, $X_{J_i}$ is a GUE matrix with dimension $N^{|J_i|}$, we have
\begin{align*}
    \bE \left[ X_{J_i} \tilde \otimes_{J_i} I_N^{\otimes (m-|J_i|)} \right] = 0
\end{align*}
and
\begin{align*}
    \bE \left[ \left( X_{J_i} \tilde \otimes_{J_i} I_N^{\otimes (m-|J_i|)} \right)^2 \right]
    = \bE \left[ \left( X_{J_i} \right)^2 \tilde \otimes_{J_i} I_N^{\otimes (m-|J_i|)} \right]
    = I_{N^{|J_i|}} \tilde  \otimes_{J_i} I_N^{\otimes (m-|J_i|)}
    = I_{N^m},
\end{align*}
where we use Lemma \ref{Lem-1} for the second formula. Then we apply Lemma \ref{lem-BBvH 7.9} to obtain
\begin{align*}
    \sp \left( X_{N,\free} \right) = \sp \left( X_\free \right).
\end{align*}
Then, the inequality \eqref{ineq-spectrum} follows from this equality and the conclusion of Step 2. The inequality \eqref{ineq-norm} follows immediately from \eqref{ineq-spectrum}.

\begin{remark}
Lemma \ref{lem-BBvH 7.9} implies that $\sp (X_{N,\free})$ does not depends on $N$. Hence, we obtain
\begin{align*}
    \bP \left( \sp(X_N) \subseteq \sp(X_{2,\free}) + C N^{-\alpha/4} \Gamma^{1/4} \Theta^{1/4} \left( \ln^{3/4} \left( dN^m \right) + t \right) [-1,1] \right) \ge 1 - e^{-t^2},
\end{align*}
and
\begin{align*}
    \bP \left( \left\| X_N \right\| > \left\| X_{2,\free} \right\| + C N^{-\alpha/4} \Gamma^{1/4} \Theta^{1/4} \left( \ln^{3/4} \left( dN^m \right) + t \right) \right)
    \le e^{-t^2}.
\end{align*}
Comparing with $X_\free$, the quantity $X_{2,\free}$ may be better in some cases as it reveals the relationship among the free semicircular family in $X_\free$.
\end{remark}

\section{Proof of Theorem \ref{Thm-2}} \label{sec:proof 2}

We use a similar idea as the proof of part b of \cite{BBvH2023}*{Theorem 2.10}. We sketch the proof below.

For any fixed $k,d$ and $d \times d$ self-adjoint deterministic matrices $B_1,\ldots,B_k$, we have $\Gamma= O_N(1)$ and $\Theta = O_N(1)$. Thus, $N^{-\alpha/4} \Gamma^{1/4} \Theta^{1/4} \ln^{3/4} (dN^m) = o_N(1)$. Hence, by Corollary \ref{Coro-2}, for any $\epsilon>0$, we have
\begin{align*}
    \sp(X_N) \subseteq \sp(X_{\free}) + (-\epsilon,\epsilon),
\end{align*}
eventually as $N \to \infty$, almost surely. Then we apply Lemma \ref{lem-BBvH 7.7} to obtain
\begin{align} \label{ineq-free-limsup}
    \limsup_{N \to \infty} \left\| P \left( X_{J_1} \tilde\otimes_{J_1} I_N^{\otimes (m-|J_1|)}, \ldots, X_{J_k} \tilde\otimes_{J_k} I_N^{\otimes (m-|J_k|)} \right) \right\|
    \le \left\| P \left( s_1,\ldots,s_k \right) \right\|
\end{align}
almost surely, for any non-commutative polynomial $P$.

We can use the computation of \eqref{ineq-v} for the case $k=d=1$ and $B_1= 1$ to obtain that for any $1 \le i \le m$,
\begin{align*}
    v \left( X_{J_i} \tilde \otimes_{J_i} I_N^{\otimes (m-|J_i|)} \right)
    \le N^{-\alpha/2}
    = o_N \left( \ln^{-3/2} \left( dN^m \right) \right).
\end{align*}
Hence, applying Lemma \ref{lem-BBvH 2.10}, we have the following almost surely convergence
\begin{align*}
    \lim_{N \to \infty} \tr \left( P \left( X_{J_1} \tilde \otimes_{J_1} I_N^{\otimes (m-|J_1|)}, \ldots, X_{J_k} \tilde \otimes_{J_k} I_N^{\otimes (m-|J_k|)} \right) \right)
    = \tau \left( P \left( s_1,\ldots,s_k \right) \right)
\end{align*}
for any non-commutative polynomial $P$. Thus, for any even positive integer $r$, we have almost surely,
\begin{align*}
    & \liminf_{N \to \infty} \left\| P \left( X_{J_1} \tilde\otimes_{J_1} I_N^{\otimes (m-|J_1|)}, \ldots, X_{J_k} \tilde\otimes_{J_k} I_N^{\otimes (m-|J_k|)} \right) \right\| \\
    & \ge \liminf_{N \to \infty} \left( \tr \left( \left| P \left( X_{J_1} \tilde \otimes_{J_1} I_N^{\otimes (m-|J_1|)}, \ldots, X_{J_k} \tilde \otimes_{J_k} I_N^{\otimes (m-|J_k|)} \right) \right|^r \right) \right)^{1/r} \\
    & = \left( \tau \left( \left| P \left( s_1,\ldots,s_k \right) \right|^r \right) \right)^{1/r},
\end{align*}
noting that $|P|^r$ is a non-commutative polynomial. Hence,
\begin{align} \label{ineq-free-liminf}
    & \liminf_{N \to \infty} \left\| P \left( X_{J_1} \tilde\otimes_{J_1} I_N^{\otimes (m-|J_1|)}, \ldots, X_{J_k} \tilde\otimes_{J_k} I_N^{\otimes (m-|J_k|)} \right) \right\| \nonumber \\
    & \ge \lim_{r \to \infty} \left( \tau \left( \left| P \left( s_1,\ldots,s_k \right) \right|^r \right) \right)^{1/r}
    = \left\| P \left( s_1,\ldots,s_k \right) \right\|.
\end{align}

The proof is concluded by combining \eqref{ineq-free-limsup} and \eqref{ineq-free-liminf}.

\begin{remark}
For $1 \le i \le k$, we define
\begin{align*}
    S_{J_i} = 2^{-|J_i|/2} \left( \sum_{l=1}^{2^{|J_i|}} E_{ll}^{(2^{|J_i|})} \otimes s_{i,l} + \sum_{1 \le r < s \le 2^{|J_i|}} \left( E_{rs}^{(2^{|J_i|})} \otimes s_{i,rs} + \tilde E_{rs}^{(2^{|J_i|})} \otimes s_{i,rs}' \right) \right),
\end{align*}
where $\{s_{i,l}, s_{i,rs}, s_{i,rs}': 1 \le i \le k, 1 \le l \le 2^{|J_i|}, 1 \le r < s \le 2^{|J_i|}\}$ is a family of free semicircular elements. Then
\begin{align*}
    X_{2,\free} = \sum_{i=1}^k B_i \otimes \left( S_{J_i} \tilde \otimes_{J_i} I_2^{\otimes (m-|J_i|)} \right).
\end{align*}
By Lemma \ref{lem-BBvH 7.9}, one can obtain that for any non-commutative polynomial $P$,
\begin{align*}
    \lim_{N \to \infty} \left\| P \left( X_{J_1} \tilde\otimes_{J_1} I_N^{\otimes (m-|J_1|)}, \ldots, X_{J_k} \tilde\otimes_{J_k} I_N^{\otimes (m-|J_k|)} \right) \right\|
    = \left\| P \left( S_{J_1},\ldots,S_{J_k} \right) \right\|,
\end{align*}
almost surely.
\end{remark}

\section{Open questions and perspectives}

In this paper, we proved that strong asymptotic freeness holds in a case beyond Haagerup and Thorbj{\o}rnsen's result (\cite{Haagerup2005}), after introducing interaction -- in the sense of quantum systems -- with independent systems of small size. 
Our results leave a handful of questions open, which we list up below: 

(1) Can one handle more general dimensions (less constraint on the size of tensors)? %if possible including all tensors of same size
Let us point out here that a very recent preprint \cite{vH2024} developed new highly original moment techniques to compute norm convergence. At first sight it was not clear though that it could allow us to make straightforward progress towards understanding our initial matrix  model.

(2) Can one handle the case where $X_i$ are all the same but acting on different legs, rather than i.i.d.? This question seems to be the original question from the point of view of quantum mechanics.
An application of \cite{BBvH2023} allowed us to compare with an operator algebraic object but although we suspect freeness, we were not able to confirm it.

(3) Can we handle the case where some $X_i$, $X_j$ are ``remote'', i.e. commute? That would also be used to address satisfactorily the initial problem in quantum physics. %{\red the case $|J_i|\le m/2$? maybe a link to P-T conjecture? The case $|J_i|=\{1,\ldots,m/2\}$ and $J_j=\{m/2+1,\ldots,m\}$.}

Finally, it would be interesting to consider the same problem not only for nearest-neighbor interaction, but for any kind of graph-type interaction. While this looks challenging, this is the most natural question from the point of view of epsilon-freeness, and it would allow to treat more models in quantumm mechanics, such as nearest neighbors interaction not only in one dimension, but in higher dimensions. 

%next work by Bordenave, Collins and Yuan (maybe too early to mention?)

%references to add: 
%https://www.cambridge.org/core/journals/canadian-journal-of-mathematics/article/freeness-and-the-partial-transposes-of-wishart-random-matrices/9C807D7530735A92D822D97BFB46393C
%recent preprint by van Handel et al
%eps freeness (Charlesworth collins)
%Magee Thomas (a model of strong freeness for eps-free)
%some papers of Bordenave C (that deal with tensors) in the intro. 
%Collins Yao Yuan? etc (from the intro)

\bibliographystyle{plain}
\bibliography{tensor}

\end{document}